\documentclass[11pt]{amsart}

\usepackage{hyperref}
\hypersetup{
	colorlinks=true,
	citecolor=blue,
	linkcolor=blue,
	filecolor=magenta,      
	urlcolor=cyan,
}
\usepackage{fullpage}
\usepackage{amsmath, amsthm, amssymb, graphicx, dsfont, stmaryrd, extarrows, enumitem}
\usepackage[T1]{fontenc}
\usepackage{textcomp}
\usepackage{lmodern}
\usepackage{microtype} 
\usepackage{amscd}
\usepackage{mathrsfs}
\usepackage{tikz-cd}
\usetikzlibrary{matrix,patterns}
\usepackage[colorinlistoftodos]{todonotes}
\usepackage{color}
\usepackage{bbm}
\usepackage{verbatim}
\usepackage[mathscr]{euscript}

\theoremstyle{plain}
\newtheorem{thm}{Theorem}[section]
\newtheorem*{thm*}{Theorem}

\newtheorem{prop}[thm]{Proposition}
    
\newtheorem{cor}[thm]{Corollary}       
\newtheorem{lem}[thm]{Lemma}
\newtheorem{que}[thm]{Question}
\theoremstyle{definition} 
\newtheorem{defn}[thm]{Definition} 

\newtheorem{ex}[thm]{Example}

\newtheorem{rem}[thm]{Remark}

\newcommand{\msf}[1]{\mathsf{#1}}

\newcommand{\mrm}[1]{\mathrm{#1}}
\newcommand{\mbb}[1]{\mathbb{#1}}
\newcommand{\p}{\mathfrak{p}}
\newcommand{\q}{\mathfrak{q}}
\newcommand{\m}{\mathfrak{m}}
\newcommand{\V}{\mathcal{V}}
\renewcommand{\k}{\Bbbk}
\newcommand{\1}{\mathds{1}}
\newcommand{\G}{\mathcal{G}}
\newcommand{\downcl}{\land}
\newcommand{\upcl}{\lor\!}
\author{Liran Shaul}
\address[Shaul]{Department of Algebra, Faculty of Mathematics and Physics, Charles University in Prague, Sokolovsk\'{a} 83, 186 75 Praha, Czech Republic}
\email{shaul@karlin.mff.cuni.cz}

\author{Jordan Williamson}
\address[Williamson]{Department of Algebra, Faculty of Mathematics and Physics, Charles University in Prague, Sokolovsk\'{a} 83, 186 75 Praha, Czech Republic}
\email{williamson@karlin.mff.cuni.cz}

\begin{document}

\title{Lifting (co)stratifications between tensor triangulated categories}% with applications to commutative DG-rings}
\maketitle 
\begin{abstract}
We give necessary and sufficient conditions for stratification and costratification to descend along a coproduct preserving, tensor-exact $R$-linear functor between $R$-linear tensor-triangulated categories which are rigidly-compactly generated by their tensor units. We then apply these results to non-positive commutative DG-rings and connective ring spectra. In particular, this gives a support-theoretic classification of (co)localizing subcategories, and thick subcategories of compact objects of the derived category of a non-positive commutative DG-ring with finite amplitude,
and provides a formal justification for the principle that the space associated to an eventually coconnective derived scheme is its underlying classical scheme.   
For a non-positive commutative DG-ring $A$, we also investigate whether certain finiteness conditions in $\msf{D}(A)$ (for example, proxy-smallness) can be reduced to questions in the better understood category $\msf{D}(H^0A)$.
\end{abstract}

\section{Introduction}
Given a triangulated category $\msf{T}$, it is generally hopeless to ask for a complete classification of its objects up to isomorphism. Instead, one can ask for a classification of its objects up to extensions, retracts and (co)products, and developing a support theory for $\msf{T}$ allows one to approach such classifications. This idea of classifying the thick, localizing and colocalizing subcategories of  triangulated categories began in chromatic homotopy theory with the classification of the thick subcategories of compact objects in the stable homotopy category by Hopkins-Smith~\cite{HopkinsSmith}. Hopkins~\cite{Hopkins} then transported this idea into algebra, providing a classification of the thick subcategories of the perfect complexes in the derived category $\msf{D}(R)$ of a commutative noetherian ring, and Neeman~\cite{Neeman1992, Neeman2011} extended this to a classification of (co)localizing subcategories of $\msf{D}(R)$. 

Benson-Iyengar-Krause~\cite{BIK08, BIK11b, BIK11, BIK12} then developed an abstract support theory for triangulated categories with the action of a ring $R$, leading to a notion of (co)stratification which provides a classification of thick, localizing and colocalizing subcategories, together with many more important consequences. They then applied this theory in the modular representation theory of finite groups~\cite{BIK11b}. 

In this paper, we investigate when stratification and costratification descend along an exact functor of tensor-triangulated categories. Given a strong symmetric monoidal, exact, coproduct-preserving functor $f^*\colon \msf{T} \to \msf{U}$ between rigidly-compactly generated tensor-triangulated categories, it follows that $f^*$ has a right adjoint $f_*$ which itself has a right adjoint $f^{(1)}$, see~\cite{BDS16}. If a ring $R$ acts on both $\msf{T}$ and $\msf{U}$ and the functor $f^*$ is suitably compatible with this action, then one can ask under what conditions $\msf{U}$ being (co)stratified by $R$ implies that $\msf{T}$ is (co)stratified by $R$.

Our first main result is the following which gives a sufficient condition for stratification and costratification to descend. Infact, under an additional mild hypothesis (which holds in the examples we study), we prove that the condition is also necessary. We direct the reader to Theorems~\ref{thm:stratdescends} and~\ref{thm:costratdescends}, and Corollaries~\ref{cor:equiv-stra} and~\ref{cor:equiv-costra} for a more detailed statement of the following main result. In this introduction we actually only state special cases of our main results, in the setting where the categories are generated by their tensor units.
\begin{thm*}
Let $f^*\colon \msf{T} \to \msf{U}$ be an exact, strong symmetric monoidal, coproduct-preserving functor between $R$-linear tensor-triangulated categories which are rigidly-compactly generated by their tensor units. Suppose that $\msf{U}$ is stratified (resp., costratified) by $R$.
\begin{enumerate}
\item If $f_*\1_\msf{U}$ builds $\1_\msf{T}$ (that is, $\1_\msf{T} \in \mrm{Loc}_\msf{T}(f_*\1_\msf{U})$), then $\msf{T}$ is stratified (resp., costratified) by the action of $R$. 
\item Moreover, if $\mrm{supp}_\msf{T}(\1_\msf{T}) \subseteq \mrm{supp}_\msf{U}(\1_\msf{U})$, then the converse holds; namely, $\msf{T}$ is stratified (resp., costratified) by $R$ if and only if $f_*\1_\msf{U}$ builds $\1_\msf{T}$.
\end{enumerate}
\end{thm*}

By the work of Benson-Iyengar-Krause~\cite{BIK11, BIK12}, the (co)stratification of a tensor-triangulated category provides a classification of the localizing subcategories, colocalizing subcategories and the thick subcategories of compact objects. As a corollary of this, we obtain bijections between the localizing subcategories of $\msf{T}$ and $\msf{U}$ (and similarly, for colocalizing subcategories and thick subcategories of compact objects). It should be noted that the bijections in the following statement are given by explicit formulas, and are not just abstract formal bijections. We direct the reader to Corollaries~\ref{cor:locofTandU} and~\ref{cor:colocofTandU} for the precise statement of the following. 
\begin{thm*}
Let $f^*\colon \msf{T} \to \msf{U}$ be an exact, strong symmetric monoidal, coproduct-preserving functor between $R$-linear tensor-triangulated categories which are rigidly-compactly generated by their tensor units. Suppose that $\msf{U}$ is costratified by $R$ and that $f_*\1_\msf{U}$ builds $\1_\msf{T}$.
Then there are bijections
\[\begin{tikzcd}\left\{\begin{tabular}{c} localizing subcategories \\ of $\msf{T}$ \end{tabular}\right\} \arrow[rr, "\cong"] \arrow[dr, "\cong"'] \arrow[dd, "\cong"'] & & \left\{\begin{tabular}{c} colocalizing subcategories \\ of $\msf{T}$ \end{tabular}\right\} \arrow[dl, "\cong"] \arrow[ll] \arrow[dd, "\cong"] \\
 & \{\text{subsets of $\mrm{supp}_\msf{T}(\1_\msf{T})$}\} \arrow[ul] \arrow[ur] \arrow[dl] \arrow[dr] & \\
\left\{\begin{tabular}{c} localizing subcategories \\ of $\msf{U}$ \end{tabular}\right\} \arrow[rr, "\cong"'] \arrow[ur, "\cong"'] \arrow[uu] & & \left\{\begin{tabular}{c} colocalizing subcategories \\ of $\msf{U}$ \end{tabular}\right\} \arrow[ul, "\cong"] \arrow[ll] \arrow[uu]
  \end{tikzcd}\]
 and 
 \[\begin{tikzcd}\left\{\begin{tabular}{c} thick subcategories \\ of compact \\ objects in $\msf{T}$ \end{tabular}\right\} \arrow[rr, "\cong"] \arrow[dr, "\cong"'] & & \left\{\begin{tabular}{c} thick subcategories \\ of compact \\ objects in $\msf{U}$ \end{tabular}\right\} \arrow[ll] \arrow[dl, "\cong"] \\
 & \left\{\begin{tabular}{c} specialization closed \\ subsets of $\mrm{supp}_\msf{T}(\1_\msf{T})$ \end{tabular}\right\}. \arrow[ul] \arrow[ur] & \end{tikzcd}\]
\end{thm*}

One could also ask for a more general version of the previous results, where the ring which acts on $\msf{T}$ can be different to the ring which acts on $\msf{U}$. Barthel-Castellana-Heard-Valenzuela~\cite{BCHV, BCHV2} have discussed descent statements for maps of commutative noetherian ring spectra where the ring acting is different, and they developed a notion of Quillen lifting to tackle this, see Remark~\ref{rem:Quillenlifting} for further discussion. In this paper, we choose to restrict to the case when the same ring acts. We emphasize that whilst this is a restrictive hypothesis, this allows us to obtain a condition which is both necessary and sufficient. Moreover, there are still interesting examples which satisfy this hypothesis.

We now turn to applications of our results. We give two main applications: to commutative DG-rings and to connective ring spectra. For a non-positive (in cohomological grading) commutative DG-ring $A$ with finite amplitude, our result applies to show that the derived category $\msf{D}(A)$ is both stratified and costratified by the action of $H^0A$, see Theorems~\ref{thm:main},~\ref{thm:consequencessupp},~\ref{thm:consequencescompacts} and~\ref{thm:consequencescosupp} in the main body of the paper for a more precise statement and additional consequences. 
\begin{thm*}
Let $A$ be a non-positive commutative DG-ring with finite amplitude and $H^0A$ noetherian. Then $\msf{D}(A)$ is stratified and costratified by the action of $H^0A$.
\end{thm*}
This theorem is a strong generalization of~\cite[Theorem 8.1]{BIK11} and~\cite[Theorem 10.3]{BIK12} which show that for a formal commutative DG-ring $A$ with $H^*A$ noetherian, the derived category $\msf{D}(A)$ is stratified and costratified by $H^*A$, see Remark~\ref{rem:formal} for more details. The assumption that $A$ has finite amplitude is necessary. In Example~\ref{ex:inf-amp} we give an example of a commutative noetherian DG-ring $A$ with infinite amplitude such that $\msf{D}(A)$ is neither stratified nor costratified by the action of $H^0A$. Commutative DG-rings of this form are the affine pieces of derived algebraic geometry,
and one may view this theorem as a formal proof for the principle that the relation between an eventually coconnective derived scheme and its underlying classical scheme is similar to the relation between a scheme and the reduced scheme associated to it.

For a non-positive commutative DG-ring $A$, using the above stratification result, we also show that the reduction functor $H^0A \otimes_A^\mrm{L} -$ and the coreduction functor $\mrm{RHom}_A(H^0A,-)$ can often be used to reduce questions in $\msf{D}(A)$ to questions in the much better understood category $\msf{D}(H^0A)$. For instance, we prove in Corollary~\ref{cor:suppdeterminesbuilding} that if $A$ is commutative with finite amplitude and $H^0A$ is noetherian, then for $M,N\in \msf{D}(A)$ it holds that $M$ builds $N$ in $\msf{D}(A)$ if and only if $H^0A \otimes_A^\mrm{L} M$ builds $H^0A \otimes_A^\mrm{L} N$ in $\msf{D}(H^0A)$. We further show in Example~\ref{ex:notbuild} that this result is false if $A$ has infinite amplitude. 
These reduction results are very powerful in reducing questions from the derived category of the DG-ring $A$ to the derived category of the ring $H^0A$.

However, there are some notable exceptions here: we show that the notions of proxy-smallness and virtual-smallness as defined by Dwyer-Greenlees-Iyengar~\cite{DGIduality, dwyer2004finiteness} cannot be determined by reduction to $\msf{D}(H^0A)$, see Theorem~\ref{thm:contradiction}. In light of recent work by Briggs-Iyengar-Letz-Pollitz~\cite{BILP} (building on work of Dwyer-Greenlees-Iyengar~\cite{dwyer2004finiteness}, Pollitz~\cite{pollitz2019derived} and Letz~\cite{Letz}) who provided a characterization of locally complete intersections in terms of proxy-smallness, this suggests that locally complete intersections between commutative DG-rings are significantly more complicated than locally complete intersections between ordinary rings. This can also already be seen through work of the first author~\cite{Shaulsmooth}, where the definition is unwieldy without a restriction to maps which are retracts of maps of flat dimension 0. 

We also apply our main result to connective ring spectra $R$, giving a necessary and sufficient condition for the homotopy category of $R$-modules to be stratified and costratified by $\pi_0R$. Given a connective ring spectrum $R$, we write $\msf{D}(R)$ for its homotopy category of modules. Applied in this setting, our main result yields the following. 
\begin{thm*}
Let $R$ be a connective commutative ring spectrum with $\pi_0R$ noetherian. Then $\msf{D}(R)$ is stratified and costratified by $\pi_0R$ if and only if $\pi_0R$ builds $R$ in $\msf{D}(R)$.
\end{thm*}
We remark that given a commutative ring spectrum $R$ with $\pi_*R$ noetherian, it is expected in general that $\msf{D}(R)$ will be stratified and costratified by $\pi_*R$ rather than $\pi_0R$. This highlights that our assumption that the same ring acts on both categories is quite restrictive in this setting. Nonetheless, the above application is still valuable since it gives a necessary and sufficient condition, and under certain nilpotency assumptions one may show that (co)stratification by $\pi_*R$ holds if and only if (co)stratification by $\pi_0R$ does, see Proposition~\ref{prop:strat-rzero} and Corollary~\ref{cor:nilpotent}.

\subsection*{Acknowledgements}
Both authors were supported by the grant GA~\v{C}R 20-02760Y from the Czech Science Foundation.
Shaul was also supported by Charles University Research Centre program No.UNCE/SCI/022.
The authors thank the referee for suggestions that helped improving this manuscript.

\section{Preliminaries} 
Here we recall the necessary preliminaries on tensor-triangulated categories, and the various notions of smallness in use throughout the paper. Let $\msf{T}$ be a triangulated category with arbitrary (small) products and coproducts. Note that under a mild hypothesis (compact generation) which we will always assume, the existence of small coproducts implies the existence of small products~\cite[Proposition 8.4.6]{Neeman}.

A full subcategory of $\msf{T}$ is said to be \emph{thick} if it is replete (closed under isomorphisms), triangulated and closed under retracts, is said to be \emph{localizing} if it is thick and closed under arbitrary coproducts, and is said to be \emph{colocalizing} if it is thick and closed under arbitrary products. Given a set $\mathcal{K}$ of objects in $\msf{T}$ we write $\mrm{Thick}_\msf{T}(\mathcal{K})$ for the smallest thick subcategory of $\msf{T}$ containing $\mathcal{K}$, $\mrm{Loc}_\msf{T}(\mathcal{K})$ for the smallest localizing subcategory of $\msf{T}$ containing $\mathcal{K}$, and $\mrm{Coloc}_\msf{T}(\mathcal{K})$ for the smallest colocalizing subcategory of $\msf{T}$ containing $\mathcal{K}$. If $\mathcal{K} = \{X\}$ consists of a single object, then we write $\mrm{Loc}_\msf{T}(X)$ for $\mrm{Loc}_\msf{T}(\{X\})$ and similarly for thick and colocalizing subcategories. 
\begin{defn}
Let $X, Y \in \msf{T}$. We say that:
\begin{itemize}
\item $X$ \emph{builds} $Y$ (or $Y$ is built from $X$) if $Y \in \mrm{Loc}_\msf{T}(X)$;
\item $X$ \emph{finitely builds} $Y$ (or $Y$ is finitely built from $X$) if $Y \in \mrm{Thick}_\msf{T}(X)$;
\item $X$ \emph{cobuilds} $Y$ (or $Y$ is cobuilt from $X$) if $Y \in \mrm{Coloc}_\msf{T}(X)$.
\end{itemize}
\end{defn}

\begin{defn}
An object $X \in \msf{T}$ is said to be \emph{small} (or \emph{compact}) if the natural map
\[\bigoplus \mrm{Hom}_\msf{T}(X, A_i) \to \mrm{Hom}_\msf{T}(X, \bigoplus A_i) \] is an isomorphism for every set $\{A_i\}$ of objects of $\msf{T}$.
We write $\msf{T}^\omega$ for the full subcategory of $\msf{T}$ consisting of the compact objects.
\end{defn}

The following definitions first appeared in~\cite{DGIduality} and~\cite{dwyer2004finiteness}.
\begin{defn}
A non-zero object $X \in \msf{T}$ is said to be:
\begin{itemize}
\item \emph{proxy-small} if there exists a small object $W \in \msf{T}$ such that $X$ finitely builds $W$ and $W$ builds $X$;
\item \emph{virtually-small} if there exists a non-zero small object $W \in \msf{T}$ such that $X$ finitely builds $W$.
\end{itemize}
We call $W$ a \emph{witness} for the fact that $X$ is proxy-small/virtually-small. Note that any proxy-small object is virtually-small.
\end{defn}

Throughout the paper, we will work with tensor-triangulated categories, that is, triangulated categories with a compatible closed symmetric monoidal structure. We write $\1_\msf{T}$ for the tensor unit, $\otimes$ for the tensor product and $F(-,-)$ for the internal hom. Moreover, we will work with rigidly-compactly generated tensor-triangulated categories; that is, those which have a set $\mathcal{G}$ of compact and dualizable objects which generate $\msf{T}$ in the sense that $\mrm{Loc}_\msf{T}(\mathcal{G}) = \msf{T}$, and for which the unit $\1_\msf{T}$ is compact. Under this assumption, the compact objects of $\msf{T}$ are precisely those which are finitely built from $\mathcal{G}$, see~\cite[Theorem 2.1.3(d)]{HPS}. 

A localizing subcategory is said to be a \emph{localizing tensor ideal} if it is closed under tensoring with arbitrary objects of $\msf{T}$ and a colocalizing subcategory $\msf{C}$ is said to be \emph{hom-closed} if for $Y \in \msf{T}$ and $Z \in \msf{C}$ we have $F(Y,Z) \in \msf{C}$. Given a set $\mathcal{K}$ of objects in $\msf{T}$ we write $\mrm{Loc}_\msf{T}^\otimes(\mathcal{K})$ for the smallest localizing tensor ideal of $\msf{T}$ containing $\mathcal{K}$, and $\mrm{Coloc}^\mrm{Hom}_\msf{T}(\mathcal{K})$ for the smallest hom-closed colocalizing subcategory of $\msf{T}$ containing $\mathcal{K}$. 

If $\G$ is a set of compact generators for $\msf{T}$, then one sees that \[\mrm{Loc}^\otimes_\msf{T}(\mathcal{K}) = \mrm{Loc}_\msf{T}(\G \otimes \mathcal{K}) \quad \mathrm{and} \quad \mrm{Coloc}^\mrm{Hom}_\msf{T}(\mathcal{K}) = \mrm{Coloc}_\msf{T}(F(\G, \mathcal{K}))\] where $\mrm{Loc}_\msf{T}(\G \otimes \mathcal{K}) = \mrm{Loc}_\msf{T}(\{G \otimes K \mid G \in \G, K \in \mathcal{K}\})$ for example. Therefore, if $\msf{T}$ is compactly generated by its tensor unit $\1_\msf{T}$, then every localizing subcategory is a localizing tensor ideal, and every colocalizing subcategory is hom-closed. 

\begin{defn}
A functor $f^*\colon \msf{T} \to \msf{U}$ between tensor-triangulated categories is called \emph{geometric} if it is exact, strong symmetric monoidal and preserves arbitrary coproducts.
\end{defn}

We note that $f^*$ need not be induced by any map $f$; this notation is merely suggestive. Geometric functors satisfy various important properties as we now recall.
\begin{lem}\label{lem:Gbuilds}
Let $f^*\colon \msf{T} \to \msf{U}$ be a geometric functor between rigidly-compactly generated tensor-triangulated categories. 
\begin{enumerate}
\item\label{item:adjoint} There is an adjoint triple 
\[\begin{tikzcd}
\msf{T} \arrow[rr, "f^*" description, yshift=4mm] \arrow[rr, "f^{(1)}" description, yshift=-4mm] & & \msf{U} \arrow[ll, "f_*" description]
\end{tikzcd} \] where the left adjoints are displayed above the respective right adjoints.
\item\label{item:proj} There is a projection formula for the adjunction $(f^*, f_*)$, that is, for any $X \in \msf{U}$ and $Y \in \msf{T}$, the natural map $f_*X \otimes Y \to f_*(X \otimes f^*Y)$ is an isomorphism. 
\item\label{item:adj} For any $X \in \msf{U}$ and $Y \in \msf{T}$, there is a natural isomorphism \[F(f_*X, Y) \simeq f_*F(X,f^{(1)}Y).\]
\item\label{item:builds} If $X$ builds (resp., finitely builds, resp., cobuilds) $Y$ in $\msf{U}$, then $f_*X$ builds (resp., finitely builds, resp., cobuilds) $f_*Y$ in $\msf{T}$.
\end{enumerate}
\end{lem}
\begin{proof}
Parts (1), (2) and (3) can be found in~\cite[Corollary 2.14 and Proposition 2.15]{BDS16}. Part (4) is an immediate consequence of the fact that $f_*$ is exact and preserves coproducts and products.
\end{proof}

\begin{ex}
For orientation, it is helpful to keep in mind the example when $f^*$ is the extension of scalars functor $B \otimes^\mrm{L}_A -\colon \msf{D}(A) \to \msf{D}(B)$ along a map of commutative rings $f\colon A \to B$. Then $f_*$ is the restriction of scalars functor, and $f^{(1)}$ is the coextension of scalars functor $\mrm{RHom}_A(B,-)$. Note that the notation differs from the standard; here, $f^*$ corresponds to the derived inverse image functor along the corresponding map $\mrm{Spec}B \to \mrm{Spec}A$ of affine schemes.
\end{ex}

Recall that an exact functor $F\colon \msf{T} \to \msf{U}$ is said to be \emph{conservative} if it reflects isomorphisms, or equivalently, if $FX \simeq 0$ implies that $X \simeq 0$. 
\begin{prop}\label{prop:conservative}
Let $f^*\colon \msf{T} \to \msf{U}$ be a geometric functor between rigidly-compactly generated tensor-triangulated categories. If $\mrm{Loc}(f_*\G_\msf{U}) = \msf{T}$ then the functors $f^*$ and $f^{(1)}$ are conservative.
\end{prop}
\begin{proof}
To show that $f^{(1)}$ is conservative, first suppose that $f^{(1)}X \simeq 0$. Therefore $\mrm{Hom}_\msf{U}^*(G, f^{(1)}X) = 0$ for all $G \in \G_\msf{U}$, so $\mrm{Hom}_\msf{T}^*(f_*G, X)= 0$ by adjunction for all $G \in \G_\msf{U}$. Since $f_*\G_\msf{U}$ generates $\msf{T}$, it follows that $X \simeq 0$.

To prove that $f^*$ is conservative, first suppose that $X \in \msf{T}$ is such that $f^*X \simeq 0$. Therefore $f_*G \otimes X \simeq f_*(G \otimes f^*X) \simeq 0$ for all $G\in \G_\msf{U}$ using the projection formula (Lemma~\ref{lem:Gbuilds}(\ref{item:proj})). Since $\1_\msf{T} \in \mrm{Loc}(f_*\G_\msf{U})$, it follows that $X \simeq 0$ as required.
\end{proof}

\begin{rem}
The hypothesis that $\mrm{Loc}(f_*\G_\msf{U}) = \msf{T}$ ensures that $f_*\G_\msf{U}$ generates $\msf{T}$. However, we do \emph{not} require that the objects of $f_*\G_\msf{U}$ are compact in $\msf{T}$; indeed, in our main example of interest this is not the case, see Remark~\ref{rem:notcompact}.
\end{rem}

\section{Descent of stratification and costratification}
\subsection{Local (co)homology and (co)support}
Recall that for a (graded) commutative noetherian ring $R$ and a triangulated category $\msf{T}$, an action of $R$ on $\msf{T}$ is a map $R \to Z^*_\msf{T}(X)$ for each $X\in \msf{T}$ where $Z^*_\msf{T}(X)$ is the graded centre of $\msf{T}$. This gives the data of a map $\phi_X\colon R \to \mrm{End}_\msf{T}^*(X)$ for each $X \in \msf{T}$ such that the $R$-actions on $\mrm{Hom}_\msf{T}^*(X,Y)$ from the left and from the right are compatible, see~\cite[\S 4]{BIK08} for more details. We call such a triangulated category \emph{$R$-linear}.

We denote by $\mrm{Spec}R$ the graded prime spectrum of $R$. 
For a set $\V \subseteq \mrm{Spec}R$, we let $\V^c$ denote the complement of $\V$.
Given a specialization closed subset $\V$ of $\mrm{Spec}R$, one can construct certain (co)localization functors of torsion, localization and completion. See \cite{BIK08, BIK12} for more details about the following constructions, noting Remark~\ref{rem:notation} where we explain the difference in notation. 

There is a localization functor $L_{\V^c}\colon \msf{T} \to \msf{T}$ defined by \[\mrm{ker}(L_{\V^c}) = \{X \in \msf{T} \mid \mrm{Hom}_\msf{T}^*(C, X)_\p =0 \text{ for all $\p \in \V^c$ and $C \in \msf{T}^\omega$}\}.\]
This has a corresponding colocalization functor $\Gamma_\V$. We define $\Gamma_\p = \Gamma_{\downcl (\p)}$ and $L_\p = L_{\upcl (\p)}$, where $\downcl(\p) = \{\q \in \mrm{Spec}R \mid \p \supseteq \q\}$ and $\upcl(\p) = \{\q \in \mrm{Spec}R \mid \p \subseteq \q\}$. For each object $X \in \msf{T}$ and each specialization closed set $\V$ of $\mrm{Spec}R$ there is a triangle \[\Gamma_\V X \to X \to L_{\V^c}X;\] we warn the reader that there is no triangle $\Gamma_\p X \to X \to L_\p X$ since $\downcl(\p)^c \neq \upcl(\p)$. The functors $\Gamma_\V$ and $L_{\V^c}$ preserve coproducts and so have right adjoints by Brown Representability, see~\cite[Theorem 8.4.4]{Neeman}. We denote the right adjoints of $\Gamma_\V$ and $L_{\V^c}$ by $\Lambda_\V$ and $V_{\V^c}$ respectively. Similarly to above, given a prime ideal $\p$ we define $\Lambda_{\p} = \Lambda_{\downcl (\p)}$ and $V_{\p} = V_{\upcl (\p)}.$ There is a triangle \[V_{\V^c}X \to X \to \Lambda_\V X\] for any $X \in \msf{T}$. We denote the essential image of $\Gamma_\p$ by $\Gamma_\p \msf{T}$, and similarly for the other (co)localization functors.

If $\msf{T}$ is tensor-triangulated and the $R$-action on $\msf{T}$ is \emph{canonical}; that is, the map $R \to \mrm{End}^*_\msf{T}(X)$ factors through $\mrm{End}^*_\msf{T}(\1_\msf{T})$, then $\Gamma_\V$ and $L_{\V^c}$ are \emph{smashing}, see~\cite[Proposition 8.1]{BIK08}. This means that for any object $X \in \msf{T}$ we have natural isomorphisms $\Gamma_\V X \simeq \Gamma_\V\1_\msf{T} \otimes X$ and $L_{\V^c}X \simeq L_{\V^c}\1_\msf{T} \otimes X$. By definition, one then notices that $\Lambda_\V X \simeq F(\Gamma_\V\1_\msf{T}, X)$ and $V_{\V^c}X \simeq F(L_{\V^c}\1_\msf{T}, X).$

\begin{ex}
Let $R$ be a commutative noetherian ring, 
and consider $\msf{D}(R)$ with the canonical $R$-action on it.
Then the functor $\Gamma_\p$ is the right derived functor of the $\p$-torsion functor,
the functor $\Lambda_\p$ is the left derived functor of the $\p$-adic completion functor,
the functor $L_\p$ is the localization functor given by $M \mapsto M_{\p}$,
and the functor $V_{\p}$ is the colocalization functor given by $M \mapsto \mrm{RHom}_R(R_{\p},M)$. For more details see~\cite[Theorem 9.1]{BIK08} and \cite{PSY}.
\end{ex}

Using the functors described above, one has notions of support and cosupport. For an object $X \in \msf{T}$, the \emph{support} of $X$ is defined by \[\mrm{supp}_\msf{T}(X) = \{\p \in \mrm{Spec}R \mid \Gamma_\p L_\p X \not\simeq 0\}\] and the \emph{cosupport} of $X$ is defined by \[\mrm{cosupp}_\msf{T}(X) =  \{\p \in \mrm{Spec}R \mid \Lambda_\p V_\p X \not\simeq 0\}.\]
 
\begin{rem}\label{rem:notation}
Our notation differs from~\cite{BIK08, BIK12} in two ways: 
\begin{enumerate}
\item The composite $\Gamma_\p L_\p = \Gamma_{\downcl(\p)}L_{\upcl(\p)}$ is denoted by $\Gamma_\p$ in~\cite{BIK08}. We instead keep the notation for the localization and torsion functors distinct. The same remark applies to the composite $\Lambda_\p V_\p$.
\item We index the localization $L_U$ on a generalization closed subset $U$; in~\cite{BIK08}, $L$ is indexed on the complementary specialization closed set. Similarly, in~\cite{BIK12} the colocalization $V$ is indexed on the complementary specialization closed set.
\end{enumerate}
\end{rem}

\begin{defn}\label{defn:compatible}
Let $\msf{T}$ and $\msf{U}$ be $R$-linear tensor-triangulated categories. An \emph{$R$-linear functor} $f^*\colon \msf{T} \to \msf{U}$ is an exact functor $f^*\colon \msf{T} \to \msf{U}$, such that the triangle
\[\begin{tikzcd}
& R \arrow[dl, "\phi_X"'] \arrow[dr, "\phi_{f^*X}"] & \\
 \mrm{End}^*_\msf{T}(X) \arrow[rr, "f^*"'] & & \mrm{End}^*_\msf{U}(f^*X)
\end{tikzcd}\]
commutes for all $X \in \msf{T}$. If moreover, the actions of $R$ on $\msf{T}$ and $\msf{U}$ are canonical, we say that $f^*\colon \msf{T} \to \msf{U}$ is a \emph{canonical $R$-linear functor}.
\end{defn}

Recall that given a geometric functor $f^*\colon \msf{T} \to \msf{U}$ between rigidly-compactly generated tensor-triangulated categories, we obtain an adjoint triple $f^* \dashv f_* \dashv f^{(1)}$ from Lemma~\ref{lem:Gbuilds}(\ref{item:adjoint}).
 \begin{lem}[{\cite[Theorem 7.7]{BIK12}}]\label{lem:commutewithL}
Let $f^*\colon \msf{T} \to \msf{U}$ be a geometric $R$-linear functor between rigidly-compactly generated tensor-triangulated categories. Let $\V$ be a specialization closed subset of $\mrm{Spec}R$, $X \in \msf{T}$ and $Y \in \msf{U}$. There are natural isomorphisms: 
\[
\begin{array}{ccc}
\Gamma_\V (f^*X) \simeq f^*\Gamma_\V X & & \Gamma_\V (f_*Y) \simeq f_*\Gamma_\V Y \\
L_{\V^c} (f^*X) \simeq f^*L_{\V^c} X & & L_{\V^c} (f_*Y) \simeq f_*L_{\V^c} Y \\ 
\Lambda_\V (f^{(1)}X) \simeq f^{(1)}\Lambda_\V X & & \Lambda_\V (f_*Y) \simeq f_*\Lambda_\V Y \\
V_{\V^c} (f^{(1)}X) \simeq f^{(1)}V_{\V^c} X & &V_{\V^c} (f_*Y) \simeq f_*V_{\V^c} Y. 
\end{array}\]
\end{lem}

\begin{cor}\label{cor:supp}
Let $f^*\colon \msf{T} \to \msf{U}$ be a geometric $R$-linear functor between rigidly-compactly generated tensor-triangulated categories. If $\mrm{Loc}_\msf{T}(f_*\G_\msf{U}) = \msf{T}$, then for all $X \in \msf{T}$, we have \[\mrm{supp}_\msf{T}(X) = \mrm{supp}_\msf{U}(f^*X) \quad \text{and} \quad  \mrm{cosupp}_\msf{T}(X) = \mrm{cosupp}_\msf{U}(f^{(1)}X).\]
\end{cor}
\begin{proof}
This is a straightforward consequence of Proposition~\ref{prop:conservative} and Lemma~\ref{lem:commutewithL}.
\end{proof}

The next lemma was inspired by \cite[Proposition 2.7]{BSW2013}.

\begin{lem}\label{lem:r-zero}
Let $R$ be a graded-commutative noetherian ring, and denote by $R^0$ its subring of homogeneous elements of degree $0$,
and by $\varphi\colon R^0\to R$ the inclusion map.
Suppose that each homogeneous element $a\in R$ such that $|a|\ne 0$ is a nilpotent element.
\begin{enumerate}
\item The map $\mrm{Spec}R \to \mrm{Spec}R^0$ given by $\p \mapsto \varphi^{-1}(\p) = \p\cap R^0$ is bijective.
\item Let $\msf{T}$ be a $R$-linear triangulated category.
Then, considering $\msf{T}$ as a $R^0$-linear triangulated category by restricting the $R$-action along $\varphi$,
for any specialization closed subset $\V$ of $\mrm{Spec}R$ there are isomorphisms of functors
\[
\Gamma_{\V} \cong \Gamma_{\varphi^{-1}(\V)}, \quad \Lambda_{\V} \cong \Lambda_{\varphi^{-1}(\V)},
\quad L_{\V^c} \cong L_{\varphi^{-1}(\V^c)}, \quad V_{\V^c}\cong V_{\varphi^{-1}(\V^c)}.
\]
\end{enumerate}
\end{lem}
\begin{proof}
To prove (1) observe that if $\p\in \mrm{Spec}R$,
since $\p$ contains all nilpotent elements, 
by our assumption on $R$ it must contain all homogeneous elements of non-zero degree.
Since $\p$ is homogeneous, this implies that it is uniquely determined by $\p\cap R^0$,
so the map $\mrm{Spec}R \to \mrm{Spec}R^0$ is a bijection.
Now (2) follows from (1) and \cite[Theorem 7.7]{BIK12}, applied to the identity functor $(\msf{T},R^0) \to (\msf{T},R)$.
\end{proof}

\subsection{Descent of stratification} In this section we give a sufficient condition for stratification to descend along a geometric functor of rigidly-compactly generated tensor-triangulated categories.

Recall from~\cite{BIK11} that an $R$-linear tensor-triangulated category $\msf{T}$ is said to be \emph{stratified by $R$} if the following minimality condition holds: for each $\p \in \mrm{Spec}R$, the localizing tensor ideal $\Gamma_\p L_\p \msf{T}$ is either zero or minimal (i.e., it has no proper non-zero localizing tensor ideals).
Note that the minimality condition can be restated as for every non-zero $X,Y \in \Gamma_\p L_\p \msf{T}$, we have $\mrm{Loc}^\otimes_\msf{T}(X) = \mrm{Loc}^\otimes_\msf{T}(Y)$. We also note that since we work in the tensor-triangulated setting, the local-to-global principle automatically holds~\cite[Theorem 8.6]{BIK12}.

\begin{thm}\label{thm:stratdescends}
Let $f^*\colon \msf{T} \to \msf{U}$ be a geometric canonical $R$-linear functor between rigidly-compactly generated tensor-triangulated categories. Suppose that $\mrm{Loc}_\msf{T}(f_*\G_\msf{U}) = \msf{T}$ and that $\msf{U}$ is stratified by $R$. Then $\msf{T}$ is stratified by the action of $R$.
\end{thm}
\begin{proof}
Given a non-zero $X \in \Gamma_\p L_\p \msf{T}$, it is sufficient to show that $\mrm{Loc}^\otimes_\msf{T}(X) = \mrm{Loc}^\otimes_\msf{T}(\Gamma_\p L_\p \1_\msf{T})$.
Firstly we have \[\mrm{Loc}^\otimes_\msf{T}(X) = \mrm{Loc}_\msf{T}(f_*\G_U \otimes X) = \mrm{Loc}_\msf{T}(f_*(\G_\msf{U} \otimes f^*X))\] using that $f_*\G_\msf{U}$ generates $\msf{T}$ and the projection formula (Lemma~\ref{lem:Gbuilds}(\ref{item:proj})). 

Since $f^*$ is conservative by Proposition~\ref{prop:conservative}, the object $f^*X$ is non-zero, and is in $\Gamma_\p L_\p \msf{U}$ by Lemma~\ref{lem:commutewithL}. Therefore as $\msf{U}$ is stratified by the action of $R$, we have that \[\mrm{Loc}_\msf{U}(\G_\msf{U} \otimes f^*X) = \mrm{Loc}^\otimes_\msf{U}(f^*X) = \mrm{Loc}^\otimes_\msf{U}(\Gamma_\p L_\p \1_\msf{U}) = \mrm{Loc}_\msf{U}(\G_\msf{U} \otimes \Gamma_\p L_\p \1_\msf{U}).\] Combining this with Lemma~\ref{lem:Gbuilds}(\ref{item:builds}) we obtain that $\mrm{Loc}_\msf{T}(f_*(\G_\msf{U} \otimes f^*X)) = \mrm{Loc}_\msf{T}(f_*(\G_\msf{U} \otimes \Gamma_\p L_\p\1_\msf{U}))$. Since $f^*$ is strong monoidal, we have $f^*\1_\msf{T} \simeq \1_\msf{U}$, so that \[\mrm{Loc}_\msf{T}(f_*(\G_\msf{U} \otimes \Gamma_\p L_\p\1_\msf{U})) = \mrm{Loc}_\msf{T}(f_*(\G_\msf{U} \otimes f^*\Gamma_\p L_\p\1_\msf{T})) = \mrm{Loc}_\msf{T}(f_*\G_\msf{U} \otimes \Gamma_\p L_\p \1_\msf{T}) = \mrm{Loc}^\otimes_\msf{T}(\Gamma_\p L_\p \1_\msf{T})\]
by the projection formula (Lemma~\ref{lem:Gbuilds}(\ref{item:proj})). Combining these we see that $\mrm{Loc}^\otimes_\msf{T}(X) = \mrm{Loc}^\otimes_\msf{T}(\Gamma_\p L_\p \1_\msf{T})$ as required.
\end{proof}

We now give some consequences of the previous result.
\begin{cor}\label{cor:locofTandU}
Let $f^*\colon \msf{T} \to \msf{U}$ be a geometric canonical $R$-linear functor between rigidly-compactly generated tensor-triangulated categories. Suppose that $\mrm{Loc}_\msf{T}(f_*\G_\msf{U}) = \msf{T}$, $f_*$ is conservative and that $\msf{U}$ is stratified by $R$. Then there is a bijection \[\{\mrm{localizing~tensor~ideals~of~} \msf{T}\} \xlongleftrightarrow{\cong} \{\mrm{localizing~tensor~ideals~of~} \msf{U}\} \] given by $\msf{L} \mapsto \mrm{Loc}^\otimes_\msf{U}(f^*X \mid X \in \msf{L}).$ Moreover, if $\msf{T}$ and $\msf{U}$ are noetherian (i.e., $\mrm{End}^*_\msf{T}(\1_\msf{T})$ and $\mrm{End}^*_\msf{U}(\1_\msf{U})$ are finitely generated $R$-modules), then there is a bijection \[\{\mrm{thick~tensor~ideals~of~} \msf{T}^\omega\} \xlongleftrightarrow{\cong} \{\mrm{thick~tensor~ideals~of~} \msf{U}^\omega\} \] given by $\mathsf{S} \mapsto \mrm{Thick}^\otimes_\msf{U}(f^*X \mid X \in \msf{S}).$
\end{cor}
\begin{proof}
We prove the claim about localizing tensor ideals first. Since $\msf{U}$ is stratified by $R$, $\msf{T}$ is also stratified by $R$ by Theorem~\ref{thm:stratdescends}. Note that \[\mrm{supp}(\msf{T}) = \bigcup_{G \in \G_\msf{U}} \mrm{supp}_\msf{T}(f_*G) = \bigcup_{G \in \G_\msf{U}} \mrm{supp}_\msf{U}(G) = \mrm{supp}(\msf{U})\] since $f_*$ is conservative. Consider the diagram 
\[\begin{tikzcd}\left\{\begin{tabular}{c} localizing tensor \\ ideals of $\msf{T}$ \end{tabular}\right\} \arrow[rr,"\msf{L} \mapsto \mrm{Loc}^\otimes_\msf{U}(f^*X \mid X \in \msf{L})"] \arrow[dr, "\mrm{supp}"'] & & \left\{\begin{tabular}{c} localizing tensor \\ ideals of $\msf{U}$ \end{tabular}\right\} \arrow[dl, "\mrm{supp}"] \\
 & \{\text{subsets of $\mrm{supp}(\msf{T})$}\} & \end{tikzcd}\]
in which the diagonals are bijections by~\cite[Theorem 3.8]{BIK11b}, since $\msf{T}$ and $\msf{U}$ are stratified by $R$. Therefore it suffices to verify that the diagram commutes; that is for a localizing tensor ideal $\msf{L}$ of $\msf{T}$, we need to check that \[\bigcup_{Y \in \mrm{Loc}^\otimes_\msf{U}(f^*X \mid X \in \msf{L})} \mrm{supp}_\msf{U}(Y) = \bigcup_{X \in \msf{L}} \mrm{supp}_\msf{T}(X).\] The reverse inclusion is clear from Corollary~\ref{cor:supp}. For the forward inclusion, note that since the diagonals are isomorphisms, for any $Y \in \mrm{Loc}^\otimes_\msf{U}(f^*X \mid X \in \msf{L})$ we have $\mrm{supp}_\msf{U}(Y) \subseteq \bigcup_{X \in \msf{L}} \mrm{supp}_{\msf{U}}(f^*X)$. Hence by Corollary~\ref{cor:supp} the forward inclusion holds. The claim about thick tensor ideals follows similarly, using the tensor-triangulated analogue of~\cite[Theorem 6.1]{BIK11}.
\end{proof}

\begin{rem}
The hypothesis that $f_*$ is conservative is often satisfied in practice since $f_*$ frequently takes the form of a forgetful functor. We direct the reader to Lemma~\ref{lem:conservativeifunit} for a condition which guarantees this in all of the examples we study in this paper.
\end{rem}

\subsection{Descent of costratification} In this section, we give a sufficient condition for costratification to descend along a geometric functor between rigidly-compactly generated tensor-triangulated categories.

Recall from~\cite{BIK12} that an $R$-linear tensor-triangulated category $\msf{T}$ is said to be \emph{costratified by $R$} if the following minimality condition holds:
for each $\p \in \mrm{Spec}R$, the hom-closed colocalizing subcategory $\Lambda_\p V_\p \msf{T}$ is either zero or minimal (i.e., it has no proper non-zero hom-closed colocalizing subcategories). We note that the local-to-global principle holds automatically since we are in the tensor-triangulated setting~\cite[Theorem 8.6]{BIK12}.

\begin{lem}\label{lem:locimpliescoloc}
Let $\msf{T}$ be a tensor-triangulated category and let $X, Y \in \msf{T}$. If $\mrm{Loc}_\msf{T}(X) = \mrm{Loc}_\msf{T}(Y)$, then $\mrm{Coloc}_\msf{T}(F(X,Z)) = \mrm{Coloc}_\msf{T}(F(Y,Z))$ for any $Z \in \msf{T}$.
\end{lem}
\begin{proof}
This follows from the fact that $F(-,Z)$ is exact and sends coproducts to products.
\end{proof}

\begin{thm}\label{thm:costratdescends}
Let $f^*\colon \msf{T} \to \msf{U}$ be a geometric canonical $R$-linear functor between rigidly-compactly generated tensor-triangulated categories. Suppose that $\mrm{Loc}_\msf{T}(f_*\G_\msf{U}) = \msf{T}$ and that $\msf{U}$ is costratified by $R$. Then $\msf{T}$ is costratified by the action of $R$.
\end{thm}
\begin{proof}
Given any non-zero $X \in \Lambda_\p V_\p \msf{T}$ it is sufficient to show that we have \[\mrm{Coloc}^\mrm{Hom}_\msf{T}(X) = \mrm{Coloc}^\mrm{Hom}_\msf{T}(\Lambda_\p V_\p \1_\msf{T}).\] 

Firstly,
\[\mrm{Coloc}^\mrm{Hom}_\msf{T}(X) = \mrm{Coloc}_\msf{T}(F(f_*\G_\msf{U}, X)) = \mrm{Coloc}_\msf{T}(f_*F(\G_\msf{U}, f^{(1)}X))\] by Lemma~\ref{lem:Gbuilds}(\ref{item:adj}). Since $f^{(1)}$ is conservative by Proposition~\ref{prop:conservative}, we have that $f^{(1)}X$ is non-zero. As $\msf{U}$ is costratified, it follows that 
$\mrm{Coloc}_\msf{U}(F(\G_\msf{U}, f^{(1)}X)) = \mrm{Coloc}^\mrm{Hom}_\msf{U}(\Lambda_\p V_\p f^{(1)}\1_\msf{T})$ by minimality since $f^{(1)}X \in \Lambda_\p V_\p \msf{T}$ by Lemma~\ref{lem:commutewithL}. It follows that
\begin{align*}
\mrm{Coloc}_\msf{T}(f_*F(\G_\msf{U}, f^{(1)}X)) &= \mrm{Coloc}_\msf{T}(f_*F(\G_\msf{U}, f^{(1)}\Lambda_\p V_\p \1_\msf{T})) \\
&= \mrm{Coloc}_\msf{T}(F(f_*\G_\msf{U}, \Lambda_\p V_\p\1_\msf{T})) \\
&= \mrm{Coloc}^\mrm{Hom}_\msf{T}(\Lambda_\p V_\p\1_\msf{T})\end{align*} by Lemmas~~\ref{lem:Gbuilds}(\ref{item:builds}) and~\ref{lem:Gbuilds}(\ref{item:adj}) which completes the proof.
\end{proof}

\begin{rem}\label{rem:Quillenlifting}
Barthel-Castellana-Heard-Valenzuela~\cite{BCHV, BCHV2} have investigated descent properties of (co)stratification along a map $f\colon R \to S$ of commutative noetherian ring spectra, giving sufficient conditions for the (co)stratification of $\msf{D}(S)$ by $\pi_*S$ to imply (co)stratification of $\msf{D}(R)$ by $\pi_*R$. This involves a notion of Quillen lifting which ensures that prime ideals of $\pi_*S$ can be `realized' by prime ideals of $\pi_*R$. They applied this theory to deduce (co)stratification results for cochain spectra $C^*(X; \mathbb{F}_p)$ for certain spaces $X$, most notably for a large class of classifying spaces of topological groups, extending work of Benson-Iyengar-Krause~\cite{BIK12} and Benson-Greenlees~\cite{BensonGreenlees14}. We note that the setup differs for us - we are interested in cases where the same ring acts on both categories, and therefore the Quillen lifting hypothesis is not relevant for us.
We further remark that the lifting result for costratification given in \cite{BCHV} assumes that the map $f\colon R \to S$
has a retract, an assumption that rarely holds in the applications we give below.
\end{rem} 

\begin{cor}\label{cor:colocofTandU}
Let $f^*\colon \msf{T} \to \msf{U}$ be a geometric canonical $R$-linear functor between rigidly-compactly generated tensor-triangulated categories. Suppose that $\mrm{Loc}_\msf{T}(f_*\G_\msf{U}) = \msf{T}$, $f_*$ is conservative and that $\msf{U}$ is costratified by $R$. Then there is a bijection \[\{\mrm{hom}\text{-}\mrm{closed~colocalizing~subcategories~of~} \msf{T}\} \xlongleftrightarrow{\cong} \{\mrm{hom}\text{-}\mrm{closed~colocalizing~subcategories~of~} \msf{U}\} \] given by $\mathsf{C} \mapsto \mrm{Coloc}^\mrm{Hom}_\msf{U}(f^{(1)}X \mid X \in \msf{C}).$
\end{cor}
\begin{proof}
The proof is similar to the proof of Corollary~\ref{cor:locofTandU}, using~\cite[Corollary 9.2]{BIK12}.
\end{proof}

\begin{prop}\label{prop:strat-rzero}
Let $R$ be a graded-commutative noetherian ring, and denote by $R^0$ its subring of homogeneous elements of degree $0$.
Suppose that each homogeneous element $a\in R$ such that $|a|\ne 0$ is a nilpotent element.
Let $\msf{T}$ be a $R$-linear triangulated category, 
so that $\msf{T}$ is also a $R^0$-linear triangulated category by restriction of the $R$-action.
Then $\msf{T}$ is stratified (resp., costratified) by the action of $R$ if and only if $\msf{T}$ is stratified (resp., costratified) by the action of $R^0$.
\end{prop}
\begin{proof}
This follows immediately from Lemma~\ref{lem:r-zero}.
\end{proof}

\subsection{Generation by the unit}
Consider our standard setup of a geometric functor $f^*\colon \msf{T} \to \msf{U}$ between rigidly-compactly generated tensor-triangulated categories. In this subsection, we specialize our results to the case when $\msf{T}$ and $\msf{U}$ are generated by their tensor units. Under this assumption, we are able to give a converse to our main result; in addition, the examples we focus on in the remainder of this paper are all generated by their tensor units. 

\begin{lem}\label{lem:conservativeifunit}
Let $f^*\colon \msf{T} \to \msf{U}$ be a geometric functor between rigidly-compactly generated tensor-triangulated categories. Assume moreover that $\msf{T}$ and $\msf{U}$ are generated by their tensor units. Then the functor $f_*$ is conservative.
\end{lem}
\begin{proof}
Let $X \in \msf{U}$. Since $f^*$ is strong monoidal
we have $f_*X \simeq 0$ if and only if $\mrm{Hom}_\msf{T}^*(\1_\msf{T}, f_*X) = 0$
if and only if $\mrm{Hom}_\msf{U}^*(\1_\msf{U}, X) = 0$ if and only if $X \simeq 0$.
\end{proof}

We can now give converses to our main results under a mild additional assumption which holds in all of the examples we study.
\begin{cor}\label{cor:equiv-stra}
Let $f^*\colon \msf{T} \to \msf{U}$ be a geometric canonical $R$-linear functor between rigidly-compactly
generated tensor-triangulated categories which are generated by their tensor units.
Assume further that $\mrm{supp}_\msf{T}(\1_\msf{T}) \subseteq \mrm{supp}_\msf{U}(\1_\msf{U})$,
and that $\msf{U}$ is stratified by $R$.
Then the following are equivalent:
\begin{enumerate}
\item\label{item:eq-build} $f_*\1_\msf{U}$ builds $\1_\msf{T}$; 
\item\label{item:eq-stra} $\msf{T}$ is stratified by the action of $R$.
\end{enumerate}
\end{cor}
\begin{proof}
The fact that (1) implies (2) is contained in Theorem~\ref{thm:stratdescends}.
To see the converse, note that by Lemmas~\ref{lem:commutewithL} and~\ref{lem:conservativeifunit},
there is an equality $\mrm{supp}_\msf{U}(\1_\msf{U}) = \mrm{supp}_\msf{T}(f_*\1_\msf{U})$,
so the assumption that $\mrm{supp}_\msf{T}(\1_\msf{T}) \subseteq \mrm{supp}_\msf{U}(\1_\msf{U})$
implies that $\mrm{supp}_\msf{T}(\1_\msf{T}) \subseteq \mrm{supp}_\msf{T}(f_*\1_\msf{U})$.
The result now follows from~\cite[Theorem 4.2]{BIK11}.
\end{proof}

\begin{cor}\label{cor:equiv-costra}
Let $f^*\colon \msf{T} \to \msf{U}$ be a geometric canonical $R$-linear functor between rigidly-compactly
generated tensor-triangulated categories which are generated by their tensor units.
Assume further that $\mrm{supp}_\msf{T}(\1_\msf{T}) \subseteq \mrm{supp}_\msf{U}(\1_\msf{U})$,
and that $\msf{U}$ is costratified by $R$.
Then the following are equivalent:
\begin{enumerate}
\item\label{item:eq-build2} $f_*\1_\msf{U}$ builds $\1_\msf{T}$;
\item\label{item:eq-costra} $\msf{T}$ is costratified by the action of $R$.
\end{enumerate}
\end{cor}
\begin{proof}
The fact that (1) implies (2) is contained in Theorem~\ref{thm:costratdescends}.
For the converse, note that by~\cite[Theorem 9.7]{BIK12},
the assumption that $\msf{T}$ is costratified by the action of $R$ implies that
$\msf{T}$ is also stratified by the action of $R$,
so this follows from Corollary~\ref{cor:equiv-stra}.
\end{proof}

\begin{rem}
Given an $R$-linear tensor-triangulated category $\msf{T}$, by~\cite[Theorem 9.7]{BIK12} $\msf{T}$ being costratified by $R$ implies that $\msf{T}$ is stratified by $R$. However the converse is not known in general. Combining Corollaries~\ref{cor:equiv-stra} and~\ref{cor:equiv-costra} gives a family of examples when $\msf{T}$ being stratified by $R$ is equivalent to $\msf{T}$ being costratified by $R$. 
\end{rem}

One also obtains versions of Corollary~\ref{cor:locofTandU} and Corollary~\ref{cor:colocofTandU} without the conservativity assumption on $f_*$ by Lemma~\ref{lem:conservativeifunit}. We note that in this case, stratification classifies the localizing subcategories, since when the unit generates any localizing subcategory is automatically a localizing tensor ideal, and similarly for costratification.

\begin{cor}\label{cor:build}
Let $f^*\colon \msf{T} \to \msf{U}$ be a geometric canonical $R$-linear functor between rigidly-compactly
generated tensor-triangulated categories which are generated by their tensor units.
Suppose that $f_*\1_\msf{U}$ builds $\1_\msf{T}$ and that $\msf{U}$ is stratified by $R$. 
Given $X,Y \in \msf{T}$, it holds that $X$ builds $Y$ in $\msf{T}$ if and only if $f^*X$ builds $f^*Y$ in $\msf{U}$.
\end{cor}
\begin{proof}
Since $f^*$ is exact and preserves coproducts, the forward implication is clear. 
For the reverse implication, if $f^*X$ builds $f^*Y$,
since $\msf{U}$ is stratified by $R$,
this implies by \cite[Theorem 4.2]{BIK11} that $\mrm{supp}_\msf{U}(f^*Y) \subseteq \mrm{supp}_\msf{U}(f^*X)$.
Hence, by Corollary~\ref{cor:supp}, we have that 
$\mrm{supp}_\msf{T}(Y) \subseteq \mrm{supp}_\msf{T}(X)$,
and since by Theorem~\ref{thm:stratdescends}, 
$\msf{T}$ is also stratified by $R$, 
we deduce that $X$ builds $Y$ in $\msf{T}$.
\end{proof}

The proof of the next corollary is is completely analogous to the proof of Corollary~\ref{cor:build}, so we omit it. 
\begin{cor}\label{cor:cobuild}
Let $f^*\colon \msf{T} \to \msf{U}$ be a geometric canonical $R$-linear functor between rigidly-compactly
generated tensor-triangulated categories which are generated by their tensor units.
Suppose that $f_*\1_\msf{U}$ builds $\1_\msf{T}$ and that $\msf{U}$ is costratified by $R$. 
Given $X,Y \in \msf{T}$, it holds that $X$ cobuilds $Y$ in $\msf{T}$ if and only if $f^{(1)}X$ cobuilds $f^{(1)}Y$ in $\msf{U}$.
\end{cor}

\section{Non-positive DG-rings}
In this section we apply the results of the previous section to the derived category of a non-positive DG-ring $A$. In particular, we show that for a non-positive commutative DG-ring $A$ with finite amplitude such that $H^0A$ is noetherian, the derived category $\msf{D}(A)$ is stratified and costratified by the action of $H^0A$. 

\subsection{Recollections on DG-rings} 
We give a brief recap of key definitions and features of DG-rings, and refer the reader to~\cite{Yekutielibook} for more details. 

A DG-ring $A$ is a graded ring equipped with a differential $d\colon A \to A$ of degree 1, which satisfies the Leibniz rule. We emphasize that we grade cohomologically. A DG-ring $A$ is said to be commutative if $ab= (-1)^{|a||b|}ba$ for all homogeneous $a,b \in A$, and $a^2 = 0$ if $|a|$ is odd. We say that $A$ is non-positive if $A^i = 0$ for all $i>0$. The zeroth cohomology $H^0A$ of a non-positive DG-ring $A$ is a ring (which is commutative if $A$ is), and comes equipped with a map $A \to H^0A$ of DG-rings. A DG-$A$-module $M$ is a graded $A$-module with a differential of degree 1 which satisfies the Leibniz rule. The DG-$A$-modules form an abelian category and inverting quasi-isomorphisms yields the derived category $\msf{D}(A)$ which is triangulated (and moreover tensor-triangulated if $A$ is commutative). 

Let $A$ be a non-positive DG-ring. For each $n \in \mathbb{Z}$ we have so-called \emph{smart truncation} functors $\mrm{smt^{>n}}, \mrm{smt^{\leq n}}\colon \msf{D}(A) \to \msf{D}(A)$ with the property that \[H^i(\mrm{smt}^{>n}M) = \begin{cases}H^iM & i >n \\ 0 & i \leq n \end{cases} \quad \text{and} \quad H^i(\mrm{smt}^{\leq n}M) = \begin{cases}H^iM & i \leq n \\ 0 & i > n. \end{cases}\] Furthermore there is a triangle \[\mrm{smt}^{\leq n}M \to M \to \mrm{smt}^{>n}M\] in $\msf{D}(A)$. 

For $M \in \msf{D}(A)$ we define $\mrm{sup}(M) = \mrm{sup}\{i \in \mathbb{Z} \mid H^iM \neq 0\}$ and $\mrm{inf}(M) = \mrm{inf}\{i \in \mathbb{Z} \mid H^iM \neq 0\}.$ We write $\msf{D}^{+}(A)$ for the full subcategory of $\msf{D}(A)$ consisting of $M$ with $\mrm{inf}(M) > -\infty$, $\msf{D}^{-}(A)$ for the full subcategory of $M$ with $\mrm{sup}(M) < \infty$, and $\msf{D}^\mrm{b}(A)$ for the intersection $\msf{D}^{+}(A) \cap \msf{D}^{-}(A)$. If $M \in \msf{D}^b(A)$, we may define $\mrm{amp}(M) = \mrm{sup}(M) - \mrm{inf}(M) \in \mathbb{N}$ and say that $M$ has \emph{finite amplitude}.

We say that a non-positive commutative DG-ring $A$ is \emph{noetherian} if $H^0A$ is noetherian, and for all $n<0$, $H^nA$ is a finitely generated $H^0A$-module. The full subcategory of $\msf{D}(A)$ consisting of those $M$ for which each $H^nM$ is a finitely generated $H^0A$-module is denoted by $\msf{D}_\mrm{f}(A)$; we write $\msf{D}^\mrm{b}_\mrm{f}(A)$ for the intersection $\msf{D}^\mrm{b}(A) \cap \msf{D}_\mrm{f}(A)$.

\subsection{Stratification and costratification for DG-rings}
We now apply the results of the previous sections to DG-rings. The following lemma is really just a special case of Lemma~\ref{lem:Gbuilds}(\ref{item:builds}), but we record it here since we will use it throughout (and in slightly more generality than Lemma~\ref{lem:Gbuilds}(\ref{item:builds}) claims).

\begin{lem}\label{lem:buildingdescends}
Let $f\colon A \to B$ be map of DG-rings, and let $M, N \in \msf{D}(B)$. If $M$ builds (resp., finitely builds, resp., cobuilds) $N$ in $\msf{D}(B)$, then $M$ builds (resp., finitely builds, resp., cobuilds) $N$ in $\msf{D}(A)$.
\end{lem}
\begin{proof}
This follows from the fact that the restriction of scalars functor $\msf{D}(B) \to \msf{D}(A)$ is exact and preserves coproducts and products.
\end{proof}

The following proposition appears in~\cite[Lemma 1.6]{BILP}. Since it is a fundamental ingredient in our proof of (co)stratification, we recall the proof.
\begin{prop}\label{prop:builds}
Let $A$ be a non-positive DG-ring. If $M \in \msf{D}(A)$ has finite amplitude, then $H^0A$ builds $M$ in $\msf{D}(A)$. 
\end{prop}
\begin{proof}
For each $n \in \mbb{Z}$, there is a distinguished triangle \[\mrm{smt}^{\leq n}\mrm{smt}^{>n-1}M \to \mrm{smt}^{>n-1}M \to \mrm{smt}^{>n}M\] in $\msf{D}(A)$. Note that $\mrm{smt}^{\leq n}\mrm{smt}^{>n-1}M \simeq H^nM$ which is a $H^0A$-module. Therefore $H^nM$ is built from $H^0A$ in $\msf{D}(H^0A)$, and hence in $\msf{D}(A)$ by Lemma~\ref{lem:buildingdescends}. One can now proceed by a finite (reverse) induction on $\mrm{inf}M -1 \leq n \leq \mrm{sup}M$, starting with the observation that $\mrm{smt}^{>\mrm{sup}M}M = 0$ and ending with $\mrm{smt}^{>\mrm{inf}M-1}M = M$. 
\end{proof}

\begin{cor}\label{cor:builds}
Let $A$ be a non-positive DG-ring with finite amplitude. Then $H^0A$ builds $A$ in $\msf{D}(A)$. Moreover, if $A$ is noetherian and $H^0A$ is a regular ring of finite Krull dimension then $H^0A$ finitely builds $A$ in $\msf{D}(A)$ and as such is proxy-small.
\end{cor}
\begin{proof}
The first statement follows from Proposition~\ref{prop:builds}. 
If $A$ is noetherian and $H^0A$ is regular of finite Krull dimension,
then in the proof of Proposition~\ref{prop:builds}, one sees that $H^nA$ is \emph{finitely} built from $H^0A$, and then the same inductive argument shows that $A$ is finitely built from $H^0A$. It follows that $H^0A$ is proxy-small with $A$ as a witness: $H^0A$ finitely builds $A$, $A$ builds $H^0A$ (since $A$ generates $\msf{D}(A)$), and $A$ is compact.
\end{proof}

\begin{rem}\label{rem:notcompact}
Despite the previous corollary,
assuming $A$ has finite amplitude, 
$H^0A$ is not a \emph{compact} generator for $\msf{D}(A)$ unless $A \to H^0A$ is a quasi-isomorphism by~\cite[Theorem I]{Jorgenson10} and~\cite[Theorem 0.7]{Yekutieli}.
\end{rem}

\begin{thm}\label{thm:reduction_is_conservative}
Let $A$ be a non-positive DG-ring with finite amplitude. For a map $f\colon M \to N$ in $\msf{D}(A)$, the following are equivalent:
\begin{itemize}
\item[(i)] $f\colon M \to N$ is an isomorphism in $\msf{D}(A)$;
\item[(ii)] $H^0A \otimes_A^\mrm{L} f\colon H^0A \otimes_A^\mrm{L} M \to H^0A \otimes_A^\mrm{L} N$ is an isomorphism in $\msf{D}(H^0A)$;
\item[(iii)] $\mrm{RHom}_A(H^0A, f)\colon \mrm{RHom}_A(H^0A, M) \to \mrm{RHom}_A(H^0A, N)$ is an isomorphism in $\msf{D}(H^0A)$.
\end{itemize}
In other words, the reduction $H^0A \otimes_A^\mrm{L} -$ and coreduction $\mrm{RHom}_A(H^0A, -)$ are conservative.
\end{thm}
\begin{proof}
Note that the collection of objects $Z \in \msf{D}(A)$ for which the map \[Z \otimes_A^\mrm{L} f\colon Z \otimes_A^\mrm{L} M \to Z \otimes_A^\mrm{L} N\] is an isomorphism in $\msf{D}(A)$ is localizing. Since $H^0A$ and $A$ build each other by Corollary~\ref{cor:builds}, the equivalence of (i) and (ii) follows. Similarly one notes that the collection of objects $Z \in \msf{D}(A)$ for which \[\mrm{RHom}_A(Z,f)\colon \mrm{RHom}_A(Z,M) \to \mrm{RHom}_A(Z,N)\] is an isomorphism in $\msf{D}(A)$ is localizing and the equivalence of (i) and (iii) follows.
\end{proof}

\begin{rem}
In \cite[Proposition 3.1]{Yekutieli} it was shown that for any non-positive DG-ring $A$,
the functor  $H^0A \otimes_A^\mrm{L} - : \msf{D}^-(A) \to \msf{D}^-(H^0A)$ is conservative.
Similarly, in \cite[Proposition 3.4]{shaul2018injective},
it was shown that for any non-positive DG-ring $A$
the functor $\mrm{RHom}_A(H^0A, -): \msf{D}^+(A) \to \msf{D}^+(H^0A)$ is conservative.
The above result generalizes these facts to the unbounded derived category,
under the additional assumption that $A$ has finite amplitude.
\end{rem}

\begin{rem}
Henceforth we will restrict to the case when $A$ is commutative, in which case the previous result is a consequence of Proposition~\ref{prop:conservative}. Nonetheless, it holds without this hypothesis.
\end{rem}

We now want to specialize to the case when $A$ is a non-positive commutative DG-ring. Its derived category $\msf{D}(A)$ is tensor-triangulated and the extension of scalars functor $H^0A \otimes_A^\mrm{L} -:\msf{D}(A)\to\msf{D}(H^0A)$ is a geometric functor; for clarity, in the notation of the previous sections we have $f^* = H^0A \otimes_A^\mrm{L} -$, $f_*$ is restriction of scalars along $A \to H^0A$, and $f^{(1)} = \mrm{RHom}_A(H^0A,-)$. If $H^0A$ is moreover noetherian, then $\msf{D}(H^0A)$ is stratified and costratified by the canonical action of $H^0A$, and we now use the results of the previous section to show that if $A$ has finite amplitude then $\msf{D}(A)$ is stratified and costratified by $H^0A$.

Firstly, we note that (essentially by definition) the abstract torsion and completion functors arising from~\cite{BIK08} coincide with functors described by the first author in~\cite{Shaul19}. Therefore, we have concrete, calculable definitions of $\Gamma_\p$ and $\Lambda_\p$ in terms of Koszul complexes, see~\cite[Proposition 2.4]{Shaul19} for more details. 

\begin{rem}\label{rem:bigvssmall}
We warn the reader that for $M \in \msf{D}(A)$, the notion of support we use in this paper is the \emph{small support} \[\mrm{supp}_A(M) = \{\p \in \mrm{Spec}(H^0A) \mid \Gamma_\p L_\p M \not\simeq 0\}.\] In general this is different to the \emph{big support} \[\mrm{Supp}_A(M) = \{\p \in \mrm{Spec}(H^0A) \mid L_\p M \not\simeq 0\}\] as discussed in~\cite[Definition 1.10]{Shaul20}. However, there is an inclusion $\mrm{supp}_A(M) \subseteq \mrm{Supp}_A(M),$ with equality if $M \in \msf{D}^\mrm{b}_\mrm{f}(A)$.
\end{rem}

\begin{lem}\label{lem:linear}
Let $A$ be a non-positive commutative DG-ring. Then the extension of scalars functor $H^0A \otimes_A^\mrm{L} -\colon \msf{D}(A) \to \msf{D}(H^0A)$ is a canonical $H^0A$-linear functor in the sense of Definition~\ref{defn:compatible}.
\end{lem}
\begin{proof}
Note that for any $M \in \msf{D}(A)$ the diagram
\[\begin{tikzcd}
A \arrow[rr] \arrow[d] & & H^0A  \arrow[d]  \\
\mrm{RHom}_A(M,M) \arrow[rr, "H^0A \otimes_A^\mrm{L}-"'] & & \mrm{RHom}_{H^0A}(H^0A \otimes_A^\mrm{L} M,H^0A \otimes_A^\mrm{L} M)
\end{tikzcd}\]
is commutative, so this follows from applying the functor $H^0$ to this diagram.
\end{proof}

\begin{lem}\label{lem:eq-support}
Let $A$ be a non-positive commutative DG-ring with $H^0A$ noetherian.
Then there is an equality $\mrm{supp}_A(A) = \mrm{supp}_A(H^0A)$.
\end{lem}
\begin{proof}
Given $\p \in \mrm{Spec}(H^0A)$,
since the functors $\Gamma_\p$ and $L_\p$ commute with restriction of scalars by Lemma~\ref{lem:commutewithL},
we have that $\mrm{supp}_A(H^0A) = \mrm{supp}_{H^0A}(H^0A) = \mrm{Spec}(H^0A)$,
which shows that $\mrm{supp}_A(A) \subseteq \mrm{supp}_A(H^0A)$.
The converse inclusion follows from the fact that $\Gamma_\p$ and $L_\p$ are smashing.
\end{proof}

\begin{thm}\label{thm:main}
Let $A$ be a non-positive commutative DG-ring with $H^0A$ noetherian.
The derived category $\msf{D}(A)$ is stratified and costratified by the canonical action of $H^0A$
if and only if $H^0A$ builds $A$ in $\msf{D}(A)$.
In particular, this is the case if $A$ has finite amplitude.
\end{thm}
\begin{proof}
The derived category $\msf{D}(H^0A)$ is stratified and costratified by the canonical action of $H^0A$ by 
\cite[Theorem 2.8]{Neeman1992} and \cite[Corollary 2.8]{Neeman2011}.
By combining Corollaries~\ref{cor:equiv-stra} and~\ref{cor:equiv-costra}, Lemmas~\ref{lem:linear} and~\ref{lem:eq-support} and Corollary~\ref{cor:builds}, the result follows.
\end{proof}

\begin{cor}
Let $A$ be a non-positive commutative DG-ring with finite amplitude and $H^0A$ noetherian.
Then the derived category $\msf{D}(A)$ is stratified and costratified by the canonical action of $H^*A$.
\end{cor}
\begin{proof}
Note that since $A$ has finite amplitude,
any homogeneous element of $H^*A$ which does not belong to $H^0A$ is nilpotent.
Hence, this follows from Proposition~\ref{prop:strat-rzero} and Theorem~\ref{thm:main}.
\end{proof}

\begin{rem}\label{rem:formal}
The above result is a far reaching generalization of \cite[Theorem 8.1]{BIK11} 
and \cite[Theorem 10.3]{BIK12} which proved that if $A$ is a formal commutative noetherian DG-ring
then $\msf{D}(A)$ is stratified and costratified by the canonical action of $H^*A$.
We also note the related work of~\cite[Theorem 1.6]{DAS16} which shows that for a commutative DG-ring $A$ with $H^*A$ noetherian,
the derived category $\msf{D}(A)$ is stratified by $H^*A$ if each localized prime $\p (H^*A)_\p$ is generated by a finite regular sequence. 
\end{rem}

\begin{cor}
Let $A$ be a non-positive commutative DG-ring with $H^0A$ noetherian.
Suppose that $H^*A$ is generated over $H^0A$ by elements of odd degree.
Then the derived category $\msf{D}(A)$ is stratified and costratified by the canonical action of $H^*A$
if and only if $H^0A$ builds $A$ in $\msf{D}(A)$.
\end{cor}
\begin{proof}
The commutativity assumption implies that any homogeneous element in $H^*A$ of odd degree has square zero,
so it is nilpotent.
From this the result follows using Proposition~\ref{prop:strat-rzero} and Theorem~\ref{thm:main}.
\end{proof}

The work of Benson-Iyengar-Krause shows that (co)stratification implies a lot about the structure of $\msf{D}(A)$. We record these consequences in the following three theorems. First, we deal with support.

\begin{thm}\label{thm:consequencessupp}
Let $A$ be a non-positive commutative DG-ring with finite amplitude and $H^0A$ noetherian.
\begin{enumerate}
\item\label{item:loc} There is a bijection \[\{\mrm{localizing~subcategories~of~} \msf{D}(A)\} \xlongleftrightarrow{\cong} \{\mrm{subsets~of~Spec}(H^0A)\} \]
given by $\msf{L} \mapsto \bigcup_{X\in \msf{L}}\mrm{supp}_A(X)$ with inverse $V \mapsto \{M \in \msf{D}(A) \mid \mrm{supp}_A(M) \subseteq V\}$.
\item\label{item:build} Let $M, N \in \msf{D}(A)$. Then $M$ builds $N$ if and only if $\mrm{supp}_A(M) \supseteq \mrm{supp}_A(N)$.
\item\label{item:supp} Let $M, N \in \msf{D}(A)$. Then $\mrm{supp}_A(M \otimes_A^\mrm{L} N) = \mrm{supp}_A(M) \cap \mrm{supp}_A(N).$
\item\label{item:telescope} If $A$ is noetherian, then the telescope conjecture holds in $\msf{D}(A)$; that is, for any localizing subcategory $\msf{L}$ of $\msf{D}(A)$, the following are equivalent: 
\begin{itemize}
\item[(a)] $\msf{L}$ is generated by compact objects of $\msf{D}(A)$;
\item[(b)] the associated localization functor $L$ is smashing;
\item[(c)] the support of $\msf{L}$ is specialization closed.
\end{itemize}
\item\label{item:hzloc} There is a bijection \[\{\mrm{localizing~subcategories~of~} \msf{D}(A)\} \xlongleftrightarrow{\cong} \{\mrm{localizing~subcategories~of~}\msf{D}(H^0A)\} \] given by $\msf{L} \mapsto \mrm{Loc}_{\msf{D}(H^0A)}(H^0A \otimes_A^\mrm{L} X \mid X \in \msf{L})$.
\end{enumerate}
\end{thm}
\begin{proof}
Since $\msf{D}(A)$ is stratified by the canonical action of $H^0A$ by Theorem~\ref{thm:main}, (1) follows from~\cite[Theorem 4.2]{BIK11}, and part (2) is an immediate consequence of (1) and the definition of building. Part (3) follows from~\cite[Theorem 7.3]{BIK11}, and (4) from~\cite[Theorem 6.3]{BIK11} (see also~\cite[Theorem 11.12]{BIK11b}). As the tensor unit of $\msf{D}(A)$ generates, localizing tensor ideals and localizing subcategories in $\msf{D}(A)$ are the same. Therefore, part (5) is a consequence of Corollary~\ref{cor:locofTandU}. 
\end{proof}

Next, we discuss consequences for compact objects in $\msf{D}(A)$.
 
\begin{thm}\label{thm:consequencescompacts}
Let $A$ be a non-positive commutative DG-ring with finite amplitude and $H^0A$ noetherian.
\begin{enumerate}
\item\label{item:thick} If $A$ is noetherian, there is a bijection \[\{\mrm{thick~subcategories~of~} \msf{D}(A)^\omega\} \xlongleftrightarrow{\cong} \{\mrm{specialization~closed~subsets~of~Spec}(H^0A)\} \] given by $\msf{S} \mapsto \bigcup_{X\in \msf{T}}\mrm{supp}_A(X)$ with inverse $V \mapsto \{M \in \msf{D}(A)^\omega \mid \mrm{supp}_A(M) \subseteq V\}$.
\item\label{item:fbuild} Let $M, N \in \msf{D}(A)$ be compact objects. Then $M$ finitely builds $N$ if and only if $\mrm{supp}_A(M) \supseteq \mrm{supp}_A(N)$.
\item\label{item:hzthick} If $A$ is noetherian, there is a bijection \[\{\mrm{thick~subcategories~of~} \msf{D}(A)^\omega\} \xlongleftrightarrow{\cong} \{\mrm{thick~subcategories~of~}\msf{D}(H^0A)^\omega\} \] given by $\msf{S} \mapsto \mrm{Thick}_{\msf{D}(H^0A)}(H^0A \otimes_A^\mrm{L} X \mid X \in \msf{S})$.
\end{enumerate}
\end{thm}
\begin{proof}
Since $\msf{D}(A)$ is stratified by the canonical action of $H^0A$ by Theorem~\ref{thm:main}, (1) follows from~\cite[Theorem 6.1]{BIK11}. For part (2), since $M$ is compact, by Thomason's localization theorem~\cite[Theorem 2.1(2.1.3)]{Neeman96} we have $\mrm{Thick}(M) = \mrm{Loc}(M) \cap \msf{D}(A)^\omega$. Therefore, as $N$ is also compact, $M$ finitely builds $N$ if and only if $M$ builds $N$, so the claim follows from Theorem~\ref{thm:consequencessupp}(\ref{item:build}). Finally, part (3) is a consequence of Corollary~\ref{cor:locofTandU}, as thick tensor ideals and thick subcategories are the same in $\msf{D}(A)$ since the tensor unit of $\msf{D}(A)$ generates.
\end{proof}

Here are the corresponding results for cosupport.

\begin{thm}\label{thm:consequencescosupp}
Let $A$ be a non-positive commutative DG-ring with finite amplitude and $H^0A$ noetherian.
\begin{enumerate}
\item\label{item:coloc} There is a bijection \[\{\mrm{colocalizing~subcategories~of~} \msf{D}(A)\} \xlongleftrightarrow{\cong} \{\mrm{subsets~of~Spec}(H^0A)\} \]
given by $\msf{C} \mapsto \bigcup_{X\in \msf{C}}\mrm{cosupp}_A(X)$ with inverse $V \mapsto \{M \in \msf{D}(A) \mid \mrm{cosupp}_A(M) \subseteq V\}$.
\item\label{item:locandcoloc} There is a bijection \[\{\mrm{localizing~subcategories~of~} \msf{D}(A)\} \xlongleftrightarrow{\cong} \{\mrm{colocalizing~subcategories~of~}\msf{D}(A)\} \] given by $\msf{L} \mapsto \msf{L}^{\perp} = \{M \in \msf{D}(A) \mid \mrm{RHom}_A(M,N) \simeq 0 \text{ for all $N \in \msf{L}$}\}$ with inverse $\msf{C} \mapsto {}^\perp\msf{C} = \{N \in \msf{D}(A) \mid \mrm{RHom}_A(M,N) \simeq 0 \text{ for all $M \in \msf{C}$}\}$.
\item\label{item:cobuild} Let $M, N \in \msf{D}(A)$. Then $M$ cobuilds $N$ if and only if $\mrm{cosupp}_A(M) \supseteq \mrm{cosupp}_A(N)$.
\item\label{item:cosupp} Let $M, N \in \msf{D}(A)$. Then $\mrm{cosupp}_A(\mrm{RHom}_A(M,N)) = \mrm{supp}_A(M) \cap \mrm{cosupp}_A(N)$.
\item\label{item:hzcoloc} There is a bijection \[\{\mrm{colocalizing~subcategories~of~} \msf{D}(A)\} \xlongleftrightarrow{\cong} \{\mrm{colocalizing~subcategories~of~}\msf{D}(H^0A)\} \] given by $\msf{C} \mapsto \mrm{Coloc}_{\msf{D}(H^0A)}(\mrm{RHom}_A(H^0A, X) \mid X \in \msf{C})$.
\end{enumerate}
\end{thm}
\begin{proof}
By Theorem~\ref{thm:main},  $\msf{D}(A)$ is costratified by the canonical action of $H^0A$, so (1) follows from~\cite[Proposition 5.6]{BIK12} (also see~\cite[Remark 5.7]{BIK12}), and (2) from~\cite[Corollary 9.9]{BIK12}. Part (3) follows immediately from (1) and the definition of cobuilding. Part (4) follows from~\cite[Theorem 9.5]{BIK12} since $\msf{D}(A)$ is also stratified by $H^0A$ by Theorem~\ref{thm:main}. Since the tensor unit in $\msf{D}(A)$ generates, colocalizing sucategories and hom-closed colocalizing subcategories are the same in $\msf{D}(A)$. Therefore, part (5) is a consequence of Corollary~\ref{cor:colocofTandU}.
\end{proof}

The next result may be viewed as a generalization of Remark~\ref{rem:notcompact}.
Note that we \textbf{do not} assume in it finite amplitude.

\begin{cor}\label{cor:str-implies-build}
Let $A$ be a non-positive commutative noetherian DG-ring,
and suppose that $\msf{D}(A)$ is stratified by the canonical action of $H^0A$.
Then the DG-module $H^0A$ is compact in $\msf{D}(A)$ if and only if the map of DG-rings $A\to H^0A$ is a quasi-isomorphism.
\end{cor}
\begin{proof}
By Lemma~\ref{lem:eq-support}
there is an equality $\mrm{supp}_A(A) = \mrm{supp}_A(H^0A)$.
Hence, as in~(\ref{item:thick}) and~(\ref{item:fbuild}) of Theorem~\ref{thm:consequencescompacts} above,
it follows from~\cite[Theorem 4.2]{BIK11} that if $H^0A$ is compact, 
then $H^0A$ finitely builds $A$.
This implies that $A$ has finite amplitude,
so by ~\cite[Theorem I]{Jorgenson10} or~\cite[Theorem 0.7]{Yekutieli} it follows that $A \to H^0A$ is a quasi-isomorphism.
\end{proof}

\begin{ex}\label{ex:inf-amp}
Let $\k$ be a field, and let $A = \k[t]$,
thought of as a graded ring with $t$ in (cohomological) degree $-2$.
This is a commutative noetherian DG-ring with zero differential and infinite amplitude.
As explained in~\cite[Example 7.26]{Yekutieli} and in~\cite[Example 1.7]{BILP},
it holds that $H^0A$ is a compact object in $\msf{D}(A)$,
so it follows from Corollary~\ref{cor:str-implies-build} that $\msf{D}(A)$ is not stratified by the canonical action of $H^0A$.
This in turn implies by~\cite[Theorem 9.7]{BIK12} that $\msf{D}(A)$ is also not costratified by the canonical action of $H^0A$. 
It follows from Theorem~\ref{thm:main} that $H^0A$ does not build $A$ in $\msf{D}(A)$. Infact, by~\cite[Theorem 5.2]{BIK11b} and~\cite[Theorem 10.3]{BIK12}, we know that $\msf{D}(A)$ is stratified and costratified by the action of $H^*A$.
\end{ex}

\subsection{Reduction determines buildings}
We now apply the consequences of the (co)stratification results described in Theorems~\ref{thm:consequencessupp},~\ref{thm:consequencescompacts} and~\ref{thm:consequencescosupp} to show that the reduction and coreduction functors can be used to determine a lot of structure in $\msf{D}(A)$. 

\begin{cor}\label{cor:suppdeterminesbuilding}
Let $A$ be a non-positive commutative DG-ring with finite amplitude such that $H^0A$ is noetherian, and let $M, N \in \msf{D}(A)$. 
\begin{enumerate}
\item $M$ builds $N$ in $\msf{D}(A)$ if and only if $H^0A \otimes_A^\mrm{L} M$ builds $H^0A \otimes_A^\mrm{L} N$ in $\msf{D}(H^0A)$.
\item $M$ cobuilds $N$ in $\msf{D}(A)$ if and only if $\mrm{RHom}_A(H^0A,M)$ cobuilds $\mrm{RHom}_A(H^0A,N)$ in $\msf{D}(H^0A)$.
\end{enumerate}
\end{cor}
\begin{proof}
(1) follows from Corollary~\ref{cor:build} and (2) follows from Corollary~\ref{cor:cobuild}.
\end{proof}

The next example shows that this result is false if $A$ has infinite amplitude (even if it is formal).
\begin{ex}\label{ex:notbuild}
This example is a continuation of Example~\ref{ex:inf-amp}; that is, we take $A = \k[t]$ where $t$ is in degree $-2$.
Even when $A$ has infinite amplitude, the functor $H^0A\otimes_A^\mrm{L} -:\msf{D}^{-}(A) \to \msf{D}^-(H^0A)$ is conservative by~\cite[Proposition 3.1]{Yekutieli}, so it follows as in Corollary~\ref{cor:supp}
that $\mrm{supp}_{H^0A}(H^0A\otimes_A^\mrm{L} H^0A) = \mrm{supp}_{A}(H^0A) = \mrm{supp}_{H^0A}(H^0A)$.
Since $\msf{D}(H^0A)$ is stratified by $H^0A$, 
this implies that $H^0A\otimes_A^\mrm{L} H^0A$ builds $H^0A = H^0A\otimes_A^\mrm{L} A$ in $\msf{D}(H^0A)$,
but we have seen in Example~\ref{ex:inf-amp} that $H^0A$ does not build $A$ in $\msf{D}(A)$.
\end{ex}

\begin{prop}\label{prop:compact}
Let $A$ be a non-positive commutative DG-ring $A$ with finite amplitude, and let $M \in \msf{D}(A)$. Then $M$ is compact in $\msf{D}(A)$ if and only if $H^0A \otimes_A^\mrm{L} M$ is compact in $\msf{D}(H^0A)$.
\end{prop}
\begin{proof}
The forward implication is evident since the restriction of scalars along $A \to H^0A$ preserves sums. For the reverse implication, since $H^0A \otimes_A^\mrm{L} M$ is compact it has finite flat dimension and hence $M$ also has finite flat dimension by~\cite[Theorem 4.1]{Shaul18}. Since $M$ has finite flat dimension and $A$ has finite amplitude, $M \simeq A \otimes_A^\mrm{L} M$ has finite amplitude. By~\cite[Theorem 5.11]{Yekutieli} it follows that $M$ is compact.
\end{proof}

In view of the above, it is natural to ask the following questions:
\begin{que}\label{que:proxy}
Given a DG-module $M \in \msf{D}(A)$ such that $H^0A \otimes_A^{\mrm{L}} M$ is proxy-small (resp., virtually-small) in $\msf{D}(H^0A)$,
does it follow that $M$ is proxy-small (resp., virtually-small) in $\msf{D}(A)$?
\end{que}

\begin{que}\label{que:finite-build}
Given two DG-modules $M,N \in \msf{D}(A)$ such that $H^0A \otimes_A^{\mrm{L}} M$ finitely builds $H^0A \otimes_A^{\mrm{L}} N$ in $\msf{D}(H^0A)$,
does it follow that $M$ finitely builds $N$ in $\msf{D}(A)$?
\end{que}

We remark that Question~\ref{que:finite-build} has a positive answer if both $M$ and $N$ are compact by Theorem~\ref{thm:consequencescompacts}(\ref{item:fbuild}). 
\begin{prop}
If Question~\ref{que:finite-build} has a positive answer, then Question~\ref{que:proxy} has a positive answer.
\end{prop}
\begin{proof}
According to~\cite[Proposition 4.5]{dwyer2004finiteness}, if $H^0A \otimes_A^{\mrm{L}} M$ is virtually-small,
then there exists a finite sequence $\mathbf{x}\subseteq H^0A$,
such that $H^0A \otimes_A^{\mrm{L}} M$ finitely builds the Koszul complex $K(H^0A;\mathbf{x})$.
By the base change property of the Koszul complex,
we have that $K(H^0A;\mathbf{x}) \simeq H^0A \otimes_A^\mrm{L} K(A;\mathbf{x})$.
Hence, if Question~\ref{que:finite-build} has a positive answer,
this implies that $M$ finitely builds the small DG-module $K(A;\mathbf{x})$,
so that $M$ is virtually-small. 
The corresponding result about proxy-smallness follows similarly, 
using~\cite[Proposition 4.4]{dwyer2004finiteness}.
\end{proof}

We will show that both of these questions have a negative answer.
Firstly we need some auxiliary results.

\begin{lem}\label{lem:proxy-field}
Let $\k$ be a field. Then any object of $\msf{D}(\k)$ is proxy-small (and hence also virtually-small).
\end{lem}
\begin{proof}
Given any non-zero object $M\in \msf{D}(\k)$, there exists some shift of $\k$ which is a direct summand of $M$,
and hence $M$ finitely builds $\k$, while $\k$ builds $M$, so $M$ is proxy-small.
\end{proof}

\begin{lem}\label{lem:small}
Let $A\to B$ be a map of DG-rings, 
and suppose that $B$ is small as an object of $\msf{D}(A)$.
Given $M \in \msf{D}(B)$, we have that $M$ is small in $\msf{D}(B)$
if and only if $M$ is small in $\msf{D}(A)$.
\end{lem}
\begin{proof}
This follows from the adjunction $\mrm{RHom}_A(M,-) \simeq \mrm{RHom}_B(M,\mrm{RHom}_A(B,-)).$
\end{proof}

The next result is a DG version of \cite[Proposition 7.2]{dwyer2004finiteness}.
\begin{prop}\label{prop:proxy-from-small}
Let $A\to B$ be a map of DG-rings, 
and suppose that $B$ is small in $\msf{D}(A)$.
Given $M \in \msf{D}(B)$, if $M$ is proxy-small (resp., virtually-small) in $\msf{D}(B)$,
then $M$ is proxy-small (resp., virtually-small) in $\msf{D}(A)$.
\end{prop}
\begin{proof}
We prove the claim about proxy-smallness, as the claim about virtual-smallness is proved similarly.
Let $W \in \msf{D}(B)$ be a small object,
such that $M$ finitely builds $W$ and $W$ builds $M$ in $\msf{D}(B)$.
By Lemma~\ref{lem:small} the DG-module $W$ is small in $\msf{D}(A)$,
while by Lemma~\ref{lem:buildingdescends} the DG-module $M$ finitely builds $W$ 
and the DG-module $W$ builds $M$ in $\msf{D}(A)$, 
so $M$ is proxy-small in $\msf{D}(A)$.
\end{proof}

We are now ready to show that Questions~\ref{que:proxy} and ~\ref{que:finite-build} have a negative answer:
\begin{thm}\label{thm:contradiction}
There exists a commutative noetherian DG-ring $K$ with finite amplitude,
and an object $N \in \msf{D}^{\mrm{b}}_{\mrm{f}}(K)$ such that $N$ is not virtually-small (and hence not proxy-small) in $\msf{D}(K)$,
but $H^0K \otimes^{\mrm{L}}_K N$ is proxy-small (and hence virtually-small) in $\msf{D}(H^0K)$.
\end{thm}
\begin{proof}
Let $A$ be any commutative noetherian ring which is not a locally complete intersection ring.
By~\cite[Theorem 5.4]{pollitz2019derived}, this implies that there exists some non-zero $M \in \msf{D}^{\mrm{b}}_{\mrm{f}}(A)$,
such that $M$ is not virtually-small (and hence, also not proxy-small).
For concrete examples of such $A$ and $M$, see \cite{briggs2020constructing}.

Since $M \in \msf{D}^{\mrm{b}}_{\mrm{f}}(A)$,
the small support and big support of $M$ coincide (see Remark~\ref{rem:bigvssmall}), and are equal to a non-empty closed subset of $\mrm{Spec}(A)$.
In particular, there exists some maximal ideal $\m$ in the support of $M$.
Let $\mathbf{m}$ be a finite sequence of elements in $A$ that generates $\m$,
and let $K=K(A;\mathbf{m})$ be the Koszul complex on this sequence.
Then $K$ is a commutative noetherian DG-ring with finite amplitude.

Set $N = K\otimes^{\mrm{L}}_A M$. 
We claim that $N$ is not virtually-small (and hence not proxy-small) in $\msf{D}(K)$,
but that $H^0K \otimes^{\mrm{L}}_K N$ is proxy-small (and hence virtually-small) in $\msf{D}(H^0K)$.
The latter follows from Lemma~\ref{lem:proxy-field} since $H^0K=A/\m$ is a field.
To see that $N$ is not virtually-small in $\msf{D}(K)$,
note that if it were, since $K$ is small over $A$,
this would imply by Proposition~\ref{prop:proxy-from-small} that $N = K\otimes^{\mrm{L}}_A M$ is virtually-small in $\msf{D}(A)$.
Since $A$ finitely builds $K$, it would follow that $M$ finitely builds $N = K \otimes_A^\mrm{L} M$, 
which shows $M$ is virtually-small, and this would be a contradiction.
\end{proof}

\section{Connective ring spectra}
In this section we apply our descent results for (co)stratification to the derived category of a connective ring spectrum $R$. 

Given a connective ring spectrum $R$ (that is, one for which $\pi_i R = 0$ for $i < 0$) we write $\msf{D}(R)$ for its homotopy category of $R$-modules. The zeroth homotopy $\pi_0R$ is a ring, and there is a map of ring spectra $R \to \pi_0R$. Note that we implicitly view the ordinary ring $\pi_0R$ as a ring spectrum via the Eilenberg-MacLane functor $H$; since there is a symmetric monoidal equivalence of derived categories $\msf{D}(H\pi_0R) \simeq \msf{D}(\pi_0R)$ this does no harm; see~\cite{SchwedeShipley} and~\cite{Shipley} for more details. If $R$ is commutative, then so is $\pi_0R$, and $\msf{D}(R)$ is a tensor-triangulated category which is rigidly-compactly generated by $R$. The extension of scalars functor $f^* = \pi_0R \otimes_R -$ is a geometric functor, and we have $f_*$ is restriction of scalars, and $f^{(1)} = \mrm{Hom}_R(\pi_0R, -)$. Note that these constructions are implicitly derived. 

\begin{lem}\label{lem:linearsupport}
The extension of scalars $\pi_0R \otimes_R -\colon \msf{D}(R) \to \msf{D}(\pi_0R)$ is a canonical $\pi_0R$-linear functor, and $\mrm{supp}_R(R) = \mrm{supp}_{R}(\pi_0R)$.
\end{lem}
\begin{proof}
The proof is analogous to the proofs of Lemmas~\ref{lem:linear} and~\ref{lem:eq-support}.
\end{proof}

\begin{thm}\label{thm:connstrat}
Let $R$ be a connective commutative ring spectrum with $\pi_0R$ noetherian. Then $\msf{D}(R)$ is stratified and costratified by the canonical action of $\pi_0R$ if and only if $\pi_0R$ builds $R$ in $\msf{D}(R)$.
\end{thm}
\begin{proof}
The derived category $\msf{D}(\pi_0R)$ is stratified and costratified by $\pi_0R$ by~\cite[Theorem 2.8]{Neeman1992} and~\cite[Corollary 2.8]{Neeman2011}. (More precisely, here we are using~\cite[Proposition 8.3]{BIK11} and the equivalence $\msf{D}(H\pi_0A) \simeq \msf{D}(\pi_0A)$ described above.) The result then follows from Lemma~\ref{lem:linearsupport}, together with Corollaries~\ref{cor:equiv-stra} and~\ref{cor:equiv-costra}.
\end{proof}

\begin{rem}
Since one can view non-positive (cohomologically graded) DG-rings as connective ring spectra, the previous theorem can be viewed as a generalization of Theorem~\ref{thm:main}. 
\end{rem}

If one is in the situation of Theorem~\ref{thm:connstrat}, then there are multiple consequences; instead of listing them here, we direct the reader to Theorems~\ref{thm:consequencessupp},~\ref{thm:consequencescompacts} and~\ref{thm:consequencescosupp} where we listed the consequences in the setting of a non-positive DG-ring. The consequences in the setting of connective ring spectra are completely analogous to those in the setting of non-positive DG-rings.

\begin{cor}\label{cor:nilpotent}
Let $R$ be a connective commutative ring spectrum with $\pi_0R$ noetherian. 
If any homogeneous element of $\pi_*R$ which does not have degree $0$ is nilpotent,
then $\msf{D}(R)$ is stratified and costratified by the canonical action of $\pi_*R$ if and only if $\pi_0R$ builds $R$ in $\msf{D}(R)$.
\end{cor}
\begin{proof}
This follows from Proposition~\ref{prop:strat-rzero} and Theorem~\ref{thm:connstrat}.
\end{proof}

\bibliographystyle{plain}
\bibliography{references}
\end{document}